\documentclass[12pt]{amsart}

\usepackage{amsmath, amssymb} 
\usepackage{graphicx} 
\usepackage{wrapfig}
\usepackage{float}

\newtheorem{theorem}{Theorem}[section]
\newtheorem{proposition}[theorem]{Proposition}

\numberwithin{equation}{section}

\begin{document} 
\title[]{2-knot homology and Yoshikawa move} 

\author[]{Hiroshi MATSUDA} 

\address{
Faculty of Science, 
Yamagata University, Yamagata 990-8560, JAPAN}

\email{matsuda@sci.kj.yamagata-u.ac.jp} 

\thanks{The author is partially supported by 
JSPS KAKENHI Grant number 15K04865.}

\maketitle 

\begin{abstract} 
Ng constructed an invariant of knots in ${\mathbb{R}}^3$, 
a combinatorial knot contact homology. 
Extending his study, 
we construct an invariant of surface-knots in ${\mathbb{R}}^4$ 
using marked graph diagrams. 
\end{abstract}

\section{Introduction}

Topological invariants of knots in ${\mathbb{R}}^3$ are 
constructed by Ng \cite{ng1}, \cite{ng2}, \cite{ng-frame}, \cite{ng3}, 
in a combinatorial method. 
These invariants are 
equivalent to 
the {\it knot contact homology}, 
constructed by Ekholm, Etnyre, Ng, Sullivan \cite{kch} 
in symplectic topology, 
and extended by Cieliebak, Ekholm, Latschev, Ng \cite{cieliebak}. 
The knot contact homology 
detects several classes of knots \cite{lidman-gordon}, and 
an enhancement of the knot contact homology 
is 
a complete invariant of knots \cite{ekholm-ng-shende}.

A {\it surface}-{\it knot} is a closed connected oriented surface 
embedded locally flatly in ${\mathbb{R}}^4$. 
A surface-knot is 
represented by an oriented marked graph diagram, 
an oriented knot diagram 
equipped with 4-valent marked vertices \cite{lomonaco}, \cite{yoshikawa}. 
Extending Ng's construction of knot invariants, 
we define in $\S$2 a differential graded algebra $(CY(D), \partial)$ 
associated with an oriented marked graph diagram $D$ 
that represents a surface-knot in ${\mathbb{R}}^4$. 

\begin{theorem} \label{main-theorem} 
Let $D_0$, $D_1$ denote diagrams of an oriented marked graph 
representing a surface-knot in ${\mathbb{R}}^4$. 
Then $(CY(D_0), \partial)$ is stably tame isomorphic to $(CY(D_1), \partial)$. 
\end{theorem} 

Theorem \ref{main-theorem} shows that 
the stably tame isomorphism class of $(CY(D), \partial)$ is an invariant of $F$, 
where $D$ denotes a diagram of an oriented marked graph 
representing a surface-knot $F$ in ${\mathbb{R}}^4$. 
Therefore the homology of $(CY(D), \partial)$ is an invariant of $F$, 
which we denote by $HY(F)$ and call {\it Yoshikawa homology}. 
In $\S \S$3-11, we give a proof of Theorem \ref{main-theorem}. 
In $\S$12, we show that 
the 0-dimensional homology $HY_0(F)$ distinguishes 
the spun-trefoil \cite{artin} from the 2-twist spun-trefoil \cite{zeeman}. 

\begin{theorem} \label{yoshikawa-example} 
Let $T^0(2, 3)$ denote the spun-trefoil, and 
let $T^2(2, 3)$ denote the 2-twist spun-trefoil in ${\mathbb{R}}^4$. 
Then 
$HY_0(T^0(2, 3))$ is not 
isomorphic to $HY_0(T^2(2, 3))$. 
\end{theorem}

\section{Definition}

Let $G$ denote a finite regular graph with 4-valent vertices 
$v_1, \cdots, v_q$. 
We fix a rectangular neighborhood $V_i$ of $v_i$ which is homeomorphic to 
$\{ (x, y) \in {\mathbb{R}}^2 \ | \ -1 \leq x \leq 1, -1 \leq y \leq 1 \}$, 
where $v_i$ corresponds to the origin, and 
edges incident to $v_i$ in $V_i$ correspond to 
the subset 
$\{ (x, y) \in {\mathbb{R}}^2 \ | \ -1 \leq x \leq 1, 
x^2 = y^2 \}$. 
An interval given by 
$\{ (x, 0) \in {\mathbb{R}}^2 \ | \ -1 \leq x \leq 1 \}$ on $V_i$, 
called a {\it marker}, is attached to each $v_i$. 
A {\it marked graph} is a spatial graph with markers 
in ${\mathbb{R}}^3$. 
An {\it orientation} of a marked graph $G$ is a choice of an orientation 
on each edge of $G$ such that 
orientations on four edges near a marked vertex $v_i$ 
are chosen as 
one of the right two figures in Figure \ref{crossing} 
for $i = 1, \cdots, q$. 
A marked graph is {\it orientable} if it admits an orientation. 
An {\it oriented marked graph} is an orientable marked graph 
with a fixed orientation. 
Two oriented marked graphs are {\it equivalent} 
if they are ambient isotopic in ${\mathbb{R}}^3$ 
keeping rectangular neighborhoods of marked vertices and orientations. 
A marked graph can be described by a diagram on ${\mathbb{R}}^2$, 
which is a link diagram with 4-valent vertices 
equipped with markers. 
Given a marked graph diagram $D$, 
let $L_+(D)$ (resp. $L_-(D)$) denote a link diagram obtained from $D$ by 
performing a resolution on a marked vertex $v_i$ 
along the marker 
(resp. along an arc perpendicular to the marker) in $V_i$ 
for $i = 1, \cdots, q$. 
See Figure \ref{vertex}. 
A marked graph diagram $D$ is {\it admissible} 
if both $L_+(D)$ and $L_-(D)$ represent trivial links. 
A marked graph is {\it admissible} 
if it has an admissible diagram.

\begin{figure}
\begin{center}
\includegraphics{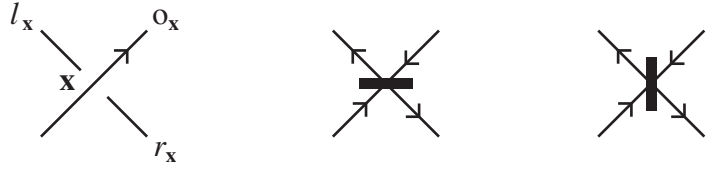}
\caption{crossing (left) and 
marked vertex (center and right)}
\label{crossing}
\end{center}
\end{figure}

Let $p \colon {\mathbb{R}}^4 \to {\mathbb{R}}$ denote 
a projection. 
We denote by ${\mathbb{R}}_t^3$ the hyperplane $p^{-1} (t)$ 
of ${\mathbb{R}}^4$, 
where $t \in {\mathbb{R}}$. 
A {\it surface}-{\it link} is a closed oriented surface 
embedded locally flatly in ${\mathbb{R}}^4$. 
Kawauchi, Shibuya, Suzuki \cite{kawauchi-shibuya-suzuki} showed that 
every surface-link ${\mathcal{L}}$ in ${\mathbb{R}}^4$ can be deformed 
by an ambient isotopy of ${\mathbb{R}}^4$ 
to a surface-link ${\mathcal{L}}'$, 
called a {\it hyperbolic splitting} of ${\mathcal{L}}$, such that the projection 
$p|_{{\mathcal{L}}'} \colon {\mathcal{L}}' \to {\mathbb{R}}$ 
satisfies the following properties: \\ 
$(1)$ 
all critical points of $p|_{{\mathcal{L}}'}$ are non-degenerate, \\ 
$(2)$ 
all the index 0 critical points (minimal points) of $p|_{{\mathcal{L}}'}$ 
lie in ${\mathbb{R}}_{-1}^3$, \\ 
$(3)$ 
all the index 1 critical points (saddle points) of $p|_{{\mathcal{L}}'}$ 
lie in ${\mathbb{R}}_0^3$, and \\ 
$(4)$ 
all the index 2 critical points (maximal points) of $p|_{{\mathcal{L}}'}$ 
lie in ${\mathbb{R}}_1^3$. \\ 
The cross-section ${\mathcal{L}}' \cap {\mathbb{R}}_0^3$ is 
a spatial 4-valent graph in ${\mathbb{R}}_0^3$. 
We assign a marker to each 4-valent vertex (saddle point) so that 
the cross-section ${\mathcal{L}}' \cap {\mathbb{R}}_\varepsilon^3$ 
(resp. ${\mathcal{L}}' \cap {\mathbb{R}}_{-\varepsilon}^3$) 
is obtained from ${\mathcal{L}}' \cap {\mathbb{R}}_0^3$ by performing 
a resolution along the marker 
(resp. along an arc perpendicular to the marker) 
at each vertex, 
where $\varepsilon$ is a small positive number. 
See Figure \ref{vertex}.  
The resulting spatial graph with markers in ${\mathbb{R}}_0^3$ 
is an admissible marked graph. 
A hyperbolic splitting 
${\mathcal{L}}'$ of ${\mathcal{L}}$ 
inherits an orientation from that of ${\mathcal{L}}$. 
We choose an orientation on each edge of 
${\mathcal{L}}' \cap {\mathbb{R}}_0^3$ 
so that it coincides with the induced orientation 
on the boundary of 
${\mathcal{L}}' \cap (\cup_{t = -\infty}^0 {\mathbb{R}}_t^3)$. 
We recall the following two theorems.

\begin{figure}
\begin{center}
\includegraphics{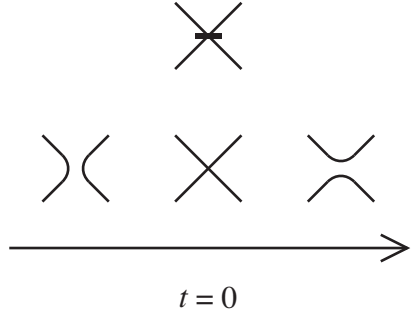}
\caption{assignment of marker}
\label{vertex}
\end{center}
\end{figure}

\begin{theorem} \label{lomonaco-yoshikawa} 
\cite{lomonaco}, \cite{yoshikawa} 
$(1)$ 
Every surface-link in ${\mathbb{R}}^4$ is represented by 
an admissible oriented marked graph in ${\mathbb{R}}^3$. \\ 
$(2)$ 
Every admissible oriented marked graph in ${\mathbb{R}}^3$ 
represents 
a surface-link 
in ${\mathbb{R}}^4$. 
\end{theorem}

\begin{theorem} \label{kearton-kurlin-swenton} 
\cite{kearton-kurlin}, \cite{swenton}, \cite{kim-joung-lee} 
Let ${\mathcal{L}}_0$ and ${\mathcal{L}}_1$ denote 
surface-links in ${\mathbb{R}}^4$. 
Let $D_0$ and $D_1$ denote diagrams of 
admissible oriented marked graphs 
that represent ${\mathcal{L}}_0$ and ${\mathcal{L}}_1$, respectively. 
Then the followings are equivalent. \\ 
$(1)$ 
${\mathcal{L}}_0$ is 
ambient isotopic to ${\mathcal{L}}_1$ in ${\mathbb{R}}^4$. \\ 
$(2)$ 
$D_1$ is obtained from $D_0$ 
by a finite sequence of 
Yoshikawa moves 
$\Omega_1$, $\Omega_1'$, $\Omega_2$, $\Omega_3$, 
$\Omega_4$, $\Omega_4'$, 
$\Omega_5$, $\Omega_6$, $\Omega_6'$, $\Omega_7$, $\Omega_8$, 
as illustrated in Figures \ref{yoshikawa1figure}, 
\ref{yoshikawa23figure}, 
\ref{yoshikawa4figure}, 
\ref{yoshikawa56figure}, 
\ref{yoshikawa78figure}. 
\end{theorem} 

Let $D$ denote a diagram of an admissible oriented marked graph 
with over/under information at each double point, called a crossing. 
Let 
${\mathbf{1}}, \cdots, {\mathbf{m}}$ 
denote labels on crossings of $D$. 
Connected components of $D$ 
are called {\it arcs}, 
labeled by 
$1, \cdots, n$. 
We notice that a marked vertex is contained in an arc. 
A unital graded algebra $CY(D)$ over ${\mathbb{Z}}$ is 
generated by the group ring ${\mathbb{Z}}[\mu, \mu^{-1}]$ 
in degree 0, 
along with the following generators: \\ 
$\{ a(i, j) \}$ in degree 0, \\ 
$\{ c({\bf x}, i) \}$, $\{ d(i, {\bf x}) \}$ 
in degree 1, \\ 
$\{ e({\bf x}, {\bf y}) \}$, $\{ f({\bf x}) \}$ 
in degree 2, \\ 
where $i \neq j \in \{ 1, \cdots, n \}$, 
${\bf x}, {\bf y} \in \{ {\bf 1}, \cdots,  {\bf m} \}$. 
We set $a(i, i) = 1 + \mu$ for $i \in \{ 1, \cdots, n \}$. 
We suppose 
that generators 
$a(i, j), c({\bf x}, i)$, $d(i, {\bf x})$, $e({\bf x}, {\bf y}), f({\bf x})$ 
do not commute with each other in $CY(D)$, and that 
$\mu$ and $\mu^{-1}$ commute with all generators in $CY(D)$. 

Let $o_{\bf x}$ denote a label assigned on the overarc 
at a crossing with label ${\bf x}$ $({\bf x} \in \{ {\bf 1}, \cdots, {\bf m} \})$, 
and let $\ell_{\bf x}$ (resp. $r_{\bf x}$) denote 
a label assigned on an underarc on the left (resp. right) 
at the crossing 
when we proceed through the crossing 
along the overarc 
following its orientation. 
See the left diagram in Figure \ref{crossing}. 
A differential $\partial$ of $CY(D)$ on generators 
is defined as follows: \\ 
$\partial a(i, j) = 0$, \\ 
$\partial c({\bf x}, i) = 
\mu a(\ell_{\bf x}, i) + a(r_{\bf x}, i) 
- a(\ell_{\bf x}, o_{\bf x}) a(o_{\bf x}, i)$, \\ 
$\partial d(i, {\bf x}) = 
a(i, \ell_{\bf x}) + \mu a(i, r_{\bf x}) 
- a(i, o_{\bf x}) a(o_{\bf x}, \ell_{\bf x})$, \\ 
$\partial e({\bf x}, {\bf y}) = 
(c({\bf x}, \ell_{\bf y}) + \mu c({\bf x}, r_{\bf y}) 
- c({\bf x}, o_{\bf y}) a(o_{\bf y}, \ell_{\bf y}))$ 
$- (\mu d(\ell_{\bf x}, {\bf y}) + d(r_{\bf x}, {\bf y}) 
- a(\ell_{\bf x}, o_{\bf x}) d(o_{\bf x}, {\bf y}))$, \\ 
$\partial f({\bf x}) = 
\mu c({\bf x}, r_{\bf x}) - \mu d(\ell_{\bf x}, {\bf x}) 
+ a(\ell_{\bf x}, o_{\bf x}) d(o_{\bf x}, {\bf x})$, \\ 
where $i \neq j \in \{ 1, \cdots, n \}, {\bf x}, {\bf y} \in \{ {\bf 1}, \cdots,  {\bf m} \}$. 
We extend the differential $\partial$ 
by linearity over ${\mathbb{Z}}$, 
and by the signed Leibniz rule: 
$\partial (v w) = (\partial v) w + (-1)^{\rm{deg} {\it v}} v (\partial w)$, 
where $v, w \in CY(D)$. 
It is straightforward to see that the equation $\partial \circ \partial = 0$ 
holds on generators in $CY(D)$.

An algebra map 
between differential graded algebras 
$\phi \colon 
({\mathbb{Z}} \langle a_1^1, \cdots, a_n^1 \rangle, \partial^1) 
\to 
({\mathbb{Z}} \langle a_1^2, \cdots, a_n^2 \rangle, \partial^2)$ 
is an {\it elementary isomorphism} 
if the followings are satisfied: \\ 
$(1)$ $\phi$ is a graded chain map, \\ 
$(2)$ $\phi(a_i^1) = \alpha a_i^2 + v$ 
for some $i \in \{ 1, \cdots, n \}$, 
where $\alpha, v \in {\mathbb{Z}} \langle a_1^2, \cdots, a_n^2 \rangle$, and 
$\alpha$ is a unit, \\ 
$(3)$ 
$\phi(a_j^1) = a_j^2$ 
for $j \neq i$. \\ 
A {\it tame isomorphism} is a composition of elementary isomorphisms. 
Let $({\mathcal{E}}^i, \partial^i)$ be the tensor algebra 
on two generators $e_1^i, e_2^i$ with 
${\rm deg} (e_1^i) - 1 = {\rm deg} (e_2^i) = i$ 
such that 
the differential is induced by $\partial^i e_1^i = e_2^i$, $\partial^i e_2^i = 0$. 
The degree-$i$ {\it algebraic stabilization} of a differential graded algebra 
$({\mathcal{A}}, \partial)$ is the coproduct of ${\mathcal{A}}$ with 
${\mathcal{E}}^i$, 
with the differential induced from $\partial$ and $\partial^i$. 
The inverse operation of the degree-$i$ algebraic stabilization 
is a degree-$i$ algebraic {\it destabilization}. 
Two differential graded algebras $(A_1, \partial_1)$ and $(A_2, \partial_2)$ 
are {\it stably tame isomorphic} 
if they are tame isomorphic 
after some number of algebraic stabilizations and destabilizations of 
$(A_1, \partial_1)$ and $(A_2, \partial_2)$.

\section{Yoshikawa move $\Omega_1$} 

\begin{figure}
\begin{center}
\includegraphics{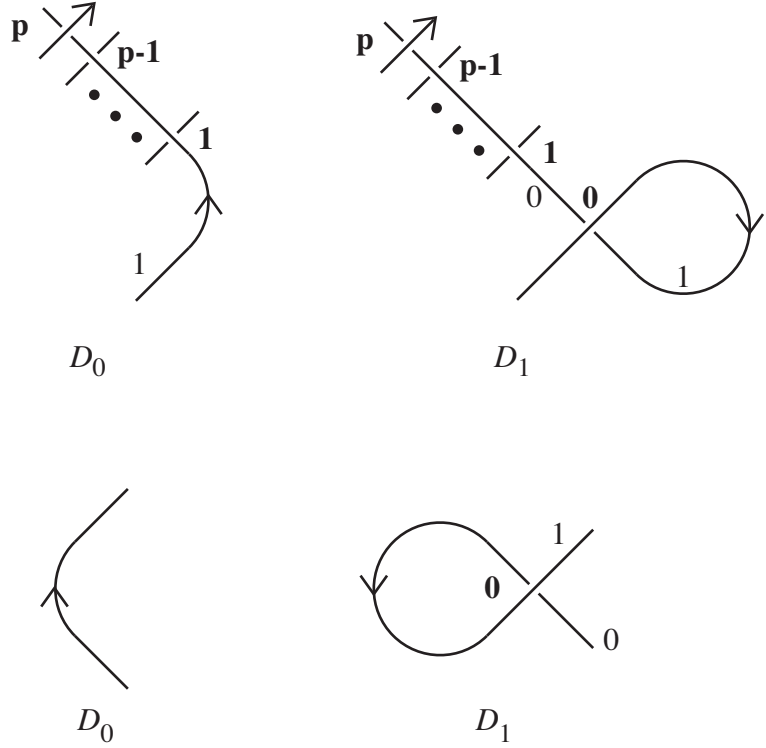}
\caption{Yoshikawa move $\Omega_1$ (upper pair) and 
Yoshikawa move $\Omega_1'$ (lower pair)}
\label{yoshikawa1figure}
\end{center}
\end{figure}

\begin{proposition} \label{yoshikawa1-1} 
Suppose that a diagram $D_1$ of an oriented marked graph $G$ 
is obtained from a diagram $D_0$ of $G$ 
by applying one Yoshikawa move $\Omega_1$, 
as illustrated in Figure \ref{yoshikawa1figure}. 
Then the differential graded algebra $(CY(D_1), \partial)$ is 
stably tame isomorphic to $(CY(D_0), \partial)$. 
\end{proposition}

\begin{proof} 
We label 
arcs of $D_0$ 
by $1, \cdots, n$, and 
crossings of $D_0$ 
by ${\bf 1}, \cdots, {\bf m}$. 
We suppose that $1$ is the label on the arc of $D_0$ 
involved in the Yoshikawa move. 
See Figure \ref{yoshikawa1figure}. 
Arcs of $D_1$ involved in the move 
are labeled by $0$ and $1$, 
and the crossing of $D_1$ created by the move 
is labeled by ${\bf 0}$. 
Let ${\bf 0}$ and ${\bf p}$ denote labels on the crossings of $D_1$ 
that are the ends of the arc with label $0$. 
Let ${\bf 1}, \cdots, {\bf p-1}$ denote labels on 
the crossings of $D_1$ on the arc with label 0. 
We suppose that 
an arc of $D_1$ with label $i$ $(i \in \{ 2, 3, \cdots, n \})$ 
corresponds to the arc of $D_0$ with label $i$, 
and that a crossing of $D_1$ with label ${\bf j}$ 
$({\bf j} \in \{ {\bf 1}, \cdots, {\bf m} \})$ 
corresponds to the crossing of $D_0$ with label ${\bf j}$. 
Without loss of generality, 
we may assume that the crossing of $D_0$ with label ${\bf p}$ 
is of positive sign. 
The diagram $D_1$ shows that 
the differential of $(CY(D_1), \partial)$ 
on generators 
is described as follows. \\ 
$\partial c({\bf 0}, i) = 
\mu a(0, i) + a(1, i) - a(0, 1) a(1, i)$, 
$\partial d(i, {\bf 0}) = 
a(i, 0) + \mu a(i, 1) - a(i, 1) a(1, 0)$, \\ 
$\partial c({\bf j}, i) = 
\mu a(\ell_{\bf j}, i) + a(r_{\bf j}, i) 
- a(\ell_{\bf j}, o_{\bf j}) a(o_{\bf j}, i)$, 
$\partial d(i, {\bf j}) = 
a(i, \ell_{\bf j}) + \mu a(i, r_{\bf j}) 
- a(i, o_{\bf j}) a(o_{\bf j}, \ell_{\bf j})$, \\ 
$\partial c({\bf p}, i) = 
\mu a(\ell_{\bf p}, i) + a(0, i) 
- a(\ell_{\bf p}, o_{\bf p}) a(o_{\bf p}, i)$, 
$\partial d(i, {\bf p}) = 
a(i, \ell_{\bf p}) + \mu a(i, 0) 
- a(i, o_{\bf p}) a(o_{\bf p}, \ell_{\bf p})$, \\ 
$\partial c({\bf k}, i) = 
\mu a(\ell_{\bf k}, i) + a(r_{\bf k}, i) 
- a(\ell_{\bf k}, 0) a(0, i)$, 
$\partial d(i, {\bf k}) = 
a(i, \ell_{\bf k}) + \mu a(i, r_{\bf k}) 
- a(i, 0) a(0, \ell_{\bf k})$, \\ 
$\partial e({\bf 0}, {\bf 0}) = 
(c({\bf 0}, 0) + \mu c({\bf 0}, 1) - c({\bf 0}, 1) a(1, 0))$ 
$- (\mu d(0, {\bf 0}) + d(1, {\bf 0}) - a(0, 1) d(1, {\bf 0}))$, \\ 
$\partial e({\bf 0}, {\bf j}) = 
(c({\bf 0}, \ell_{\bf j}) + \mu c({\bf 0}, r_{\bf j}) 
- c({\bf 0}, o_{\bf j}) a(o_{\bf j}, \ell_{\bf j}))$ 
$- (\mu d(0, {\bf j}) + d(1, {\bf j}) - a(0, 1) d(1, {\bf j}))$, \\ 
$\partial e({\bf j}, {\bf 0}) = 
(c({\bf j}, 0) + \mu c({\bf j}, 1) - c({\bf j}, 1) a(1, 0))$ 
$- (\mu d(\ell_{\bf j}, {\bf 0}) + d(r_{\bf j}, {\bf 0}) 
- a(\ell_{\bf j}, o_{\bf j}) d(o_{\bf j}, {\bf 0}))$, \\ 
$\partial e({\bf 0}, {\bf p}) = 
(c({\bf 0}, \ell_{\bf p}) + \mu c({\bf 0}, 0) 
- c({\bf 0}, o_{\bf p}) a(o_{\bf p}, \ell_{\bf p}))$ 
$- (\mu d(0, {\bf p}) + d(1, {\bf p}) - a(0, 1) d(1, {\bf p}))$, \\ 
$\partial e({\bf p}, {\bf 0}) = 
(c({\bf p}, 0) + \mu c({\bf p}, 1) - c({\bf p}, 1) a(1, 0))$ 
$- (\mu d(\ell_{\bf p}, {\bf 0}) + d(0, {\bf 0}) 
- a(\ell_{\bf p}, o_{\bf p}) d(o_{\bf p}, {\bf 0}))$, \\ 
$\partial e({\bf 0}, {\bf k}) = 
(c({\bf 0}, \ell_{\bf k}) + \mu c({\bf 0}, r_{\bf k}) 
- c({\bf 0}, 0) a(0, \ell_{\bf k}))$ 
$- (\mu d(0, {\bf k}) + d(1, {\bf k}) 
- a(0, 1) d(1, {\bf k}))$, \\ 
$\partial e({\bf k}, {\bf 0}) = 
(c({\bf k}, 0) + \mu c({\bf k}, 1) 
- c({\bf k}, 1) a(1, 0))$ 
$- (\mu d(\ell_{\bf k}, {\bf 0}) + d(r_{\bf k}, {\bf 0}) 
- a(\ell_{\bf k}, 0) d(0, {\bf 0}))$, \\ 
$\partial e({\bf j1}, {\bf j2}) = 
(c({\bf j1}, \ell_{\bf j2}) + \mu c({\bf j1}, r_{\bf j2}) 
- c({\bf j1}, o_{\bf j2}) a(o_{\bf j2}, \ell_{\bf j2}))$ 
$- (\mu d(\ell_{\bf j1}, {\bf j2}) + d(r_{\bf j1}, {\bf j2}) 
- a(\ell_{\bf j1}, o_{\bf j1}) d(o_{\bf j1}, {\bf j2}))$, \\ 
$\partial e({\bf j}, {\bf p}) = 
(c({\bf j}, \ell_{\bf p}) + \mu c({\bf j}, 0) 
- c({\bf j}, o_{\bf p}) a(o_{\bf p}, \ell_{\bf p}))$ 
$- (\mu d(\ell_{\bf j}, {\bf p}) + d(r_{\bf j}, {\bf p}) 
- a(\ell_{\bf j}, o_{\bf j}) d(o_{\bf j}, {\bf p}))$, \\ 
$\partial e({\bf p}, {\bf j}) = 
(c({\bf p}, \ell_{\bf j}) + \mu c({\bf p}, r_{\bf j}) 
- c({\bf p}, o_{\bf j}) a(o_{\bf j}, \ell_{\bf j}))$ 
$- (\mu d(\ell_{\bf p}, {\bf j}) + d(0, {\bf j}) 
- a(\ell_{\bf p}, o_{\bf p}) d(o_{\bf p}, {\bf j}))$, \\ 
$\partial e({\bf j}, {\bf k}) = 
(c({\bf j}, \ell_{\bf k}) + \mu c({\bf j}, r_{\bf k}) 
- c({\bf j}, 0) a(0, \ell_{\bf k}))$ 
$- (\mu d(\ell_{\bf j}, {\bf k}) + d(r_{\bf j}, {\bf k}) 
- a(\ell_{\bf j}, o_{\bf j}) d(o_{\bf j}, {\bf k}))$, \\ 
$\partial e({\bf k}, {\bf j}) = 
(c({\bf k}, \ell_{\bf j}) + \mu c({\bf k}, r_{\bf j}) 
- c({\bf k}, o_{\bf j}) a(o_{\bf j}, \ell_{\bf j}))$ 
$- (\mu d(\ell_{\bf k}, {\bf j}) + d(r_{\bf k}, {\bf j}) 
- a(\ell_{\bf k}, 0) d(0, {\bf j}))$, \\ 
$\partial e({\bf p}, {\bf p}) = 
(c({\bf p}, \ell_{\bf p}) + \mu c({\bf p}, 0) 
- c({\bf p}, o_{\bf p}) a(o_{\bf p}, \ell_{\bf p}))$ 
$- (\mu d(\ell_{\bf p}, {\bf p}) + d(0, {\bf p}) 
- a(\ell_{\bf p}, o_{\bf p}) d(o_{\bf p}, {\bf p}))$, \\ 
$\partial e({\bf p}, {\bf k}) = 
(c({\bf p}, \ell_{\bf k}) + \mu c({\bf p}, r_{\bf k}) 
- c({\bf p}, 0) a(0, \ell_{\bf k}))$ 
$- (\mu d(\ell_{\bf p}, {\bf k}) + d(0, {\bf k}) 
- a(\ell_{\bf p}, o_{\bf p}) d(o_{\bf p}, {\bf k}))$, \\ 
$\partial e({\bf k}, {\bf p}) = 
(c({\bf k}, \ell_{\bf p}) + \mu c({\bf k}, 0) 
- c({\bf k}, o_{\bf p}) a(o_{\bf p}, \ell_{\bf p}))$ 
$- (\mu d(\ell_{\bf k}, {\bf p}) + d(r_{\bf k}, {\bf p}) 
- a(\ell_{\bf k}, 0) d(0, {\bf p}))$, \\ 
$\partial e({\bf k1}, {\bf k2}) = 
(c({\bf k1}, \ell_{\bf k2}) + \mu c({\bf k1}, r_{\bf k2}) 
- c({\bf k1}, 0) a(0, \ell_{\bf k2}))$ 
$- (\mu d(\ell_{\bf k1}, {\bf k2}) + d(r_{\bf k1}, {\bf k2}) 
- a(\ell_{\bf k1}, 0) d(0, {\bf k2}))$, \\ 
$\partial f({\bf 0}) = 
\mu c({\bf 0}, 1) - \mu d(0, {\bf 0}) 
+ a(0, 1) d(1, {\bf 0})$, 
$\partial f({\bf j}) = 
\mu c({\bf j}, r_{\bf j}) - \mu d(\ell_{\bf j}, {\bf j}) 
+ a(\ell_{\bf j}, o_{\bf j}) d(o_{\bf j}, {\bf j})$, \\ 
$\partial f({\bf p}) = 
\mu c({\bf p}, 0) - \mu d(\ell_{\bf p}, {\bf p}) 
+ a(\ell_{\bf p}, o_{\bf p}) d(o_{\bf p}, {\bf p})$, 
$\partial f({\bf k}) = 
\mu c({\bf k}, r_{\bf k}) - \mu d(\ell_{\bf k}, {\bf k}) 
+ a(\ell_{\bf k}, 0) d(0, {\bf k})$, \\ 
where 
$i \in \{ 0, 1, \cdots, n \}$, 
${\bf j}, {\bf j1}, {\bf j2} \in \{ {\bf 1}, \cdots, {\bf p-1} \}$ and  
${\bf k}, {\bf k1}, {\bf k2} \in \{ {\bf{p+1}}, \cdots, {\bf m} \}$.

We define a tame isomorphism 
$\phi \colon (CY(D_1), \partial) \to (CY^2(D_1), \partial^2)$ 
by \\ 
$\phi(a(0, i)) = 
a^2(0, i) + a^2(1, i) - a^2(0, 1) a^2(1, i)$, 
$\phi(a(i, 0)) = 
\mu a^2(i, 0) + a^2(i, 1) - a^2(i, 1) a^2(1, 0)$, \\ 
$\phi(a(i, j)) = a^2(i, j)$, 
$\phi(c({\bf j}, i)) = c^2({\bf j}, i)$, 
$\phi(d(i, {\bf j})) = d^2(i, {\bf j})$, \\ 
$\phi(e({\bf j}, {\bf k})) = e^2({\bf j}, {\bf k})$, 
$\phi(f({\bf j})) = f^2({\bf j})$, \\ 
where $i \neq j \in \{ 1, \cdots, n \}$ and 
${\bf j}, {\bf k} \in \{ {\bf 0}, {\bf 1}, \cdots, {\bf m} \}$. 
For example, 
from the equation 
$\partial c({\bf 0}, i) = 
\mu a(0, i) + a(1, i) - a(0, 1) a(1, i)$ 
in $(CY(D_1), \partial)$, 
we obtain 
$\partial^2 c^2({\bf 0}, i) = 
\mu (a^2(0, i) + a^2(1, i) - a^2(0, 1) a^2(1, i)) 
+ a^2(1, i) 
- (a^2(0, 1) + a^2(1, 1) - a^2(0, 1) a^2(1, 1)) a^2(1, i)$ 
$= \mu a^2(0, i)$ in $(CY^2(D_1), \partial^2)$, 
where $i \in \{ 1, \cdots, n \}$. 
We notice equations 
$\partial^2 c^2({\bf 0}, i) = \mu a^2(0, i)$,  
$\partial^2 a^2(0, i) = 0$,  
and 
$\partial^2 d^2(i, {\bf 0}) = \mu a^2(i, 0)$, 
$\partial^2 a^2(i, 0) = 0$ 
in $(CY^2(D_1), \partial^2)$, 
where $i \in \{ 1, \cdots, n \}$. 

We eliminate pairs of generators 
$(c^2({\bf 0}, i), a^2(0, i))$, 
$(d^2(i, {\bf 0}), a^2(i, 0))$ 
$(i \in \{ 1, \cdots, n \})$ 
by a sequence of destabilizations 
on $(CY^2(D_1), \partial^2)$, 
and we obtain $(CY^3(D_1), \partial^3)$. 
For example, 
from the equation $\partial^2 c^2({\bf 0}, 0) = 
\mu (1 + \mu) 
+ (\mu a^2(1, 0) + (1 + \mu) - (1 + \mu) a^2(1, 0)) 
- (a^2(0, 1) + (1 + \mu) - a^2(0, 1) (1 + \mu)) 
(\mu a^2(1, 0) + (1 + \mu) - (1 + \mu) a^2(1, 0))$ 
in $(CY^2(D_1), \partial^2)$, 
we obtain 
$\partial^3 c^3({\bf 0}, 0) 
= \mu (1 + \mu) 
+ (1 + \mu) 
- (1 + \mu) (1 + \mu) 
= 0$ in $(CY^3(D_1), \partial^3)$. 
We notice equations 
$\partial^3 f^3({\bf 0}) = - \mu d^3(0, {\bf 0})$, 
$\partial^3 d^3(0, {\bf 0}) = 0$, and 
$\partial^3 e^3({\bf 0}, {\bf 0}) = c^3({\bf 0}, 0) - \mu d^3(0, {\bf 0})$, 
$\partial^3 c^3({\bf 0}, 0) = 0$  
in $(CY^3(D_1), \partial^3)$. 

We eliminate the pair of generators 
$(f^3({\bf 0}), d^3(0, {\bf 0}))$ 
by a destabilization on $(CY^3(D_1), \partial^3)$, 
and we obtain $(CY^4(D_1), \partial^4)$. 
From the equations 
$\partial^3 e^3({\bf 0}, {\bf 0}) = c^3({\bf 0}, 0) - \mu d^3(0, {\bf 0})$ 
and 
$\partial^3 c^3({\bf 0}, 0) = 0$  
in $(CY^3(D_1), \partial^3)$, 
we obtain 
$\partial^4 e^4({\bf 0}, {\bf 0}) = c^4({\bf 0}, 0)$ and 
$\partial^4 c^4({\bf 0}, 0) = 0$ in $(CY^4(D_1), \partial^4)$, 
respectively. 
We eliminate the pair of generators 
$(e^4({\bf 0}, {\bf 0}), c^4({\bf 0}, 0))$ 
by a destabilization on $(CY^4(D_1), \partial^4)$, 
and we obtain $(CY^5(D_1), \partial^5)$. 

We define a tame isomorphism 
$\phi^5 \colon (CY^5(D_1), \partial^5) \to (CY^6(D_1), \partial^6)$ 
by \\ 
$\phi^5 (c^5({\bf j}, 0)) = c^6({\bf j}, 0) + c^6({\bf j}, 1)$, 
$\phi^5 (d^5(0, {\bf j})) = d^6(0, {\bf j}) + d^6(1, {\bf j})$, \\ 
$\phi^5 (a^5(i, j)) = a^6(i, j)$, 
$\phi^5 (c^5({\bf j}, i)) = c^6({\bf j}, i)$, 
$\phi^5 (d^5(i, {\bf j})) = d^6(i, {\bf j})$, \\ 
$\phi^5 (e^5({\bf j}, {\bf 0})) = e^6({\bf j}, {\bf 0})$, 
$\phi^5 (e^5({\bf 0}, {\bf j})) = e^6({\bf 0}, {\bf j})$, 
$\phi^5 (e^5({\bf j}, {\bf k})) = e^6({\bf j}, {\bf k})$, 
$\phi^5 (f^5({\bf j})) = f^6({\bf j})$, \\ 
where $i \neq j \in \{ 1, \cdots, n \}$ and 
${\bf j}, {\bf k} \in \{ {\bf 1}, \cdots, {\bf m} \}$. 
For example, 
from the equation $\partial^5 c^5({\bf j}, 0) = 
\mu a^5(\ell_{\bf j}, 1) + a^5(r_{\bf j}, 1) 
- a^5(\ell_{\bf j}, o_{\bf j}) a^5(o_{\bf j}, 1)$ 
in $(CY^5(D_1), \partial^5)$, 
we obtain 
$\partial^6 (c^6({\bf j}, 0) + c^6({\bf j}, 1)) 
= \mu a^6(\ell_{\bf j}, 1) + a^6(r_{\bf j}, 1) 
- a^6(\ell_{\bf j}, o_{\bf j}) a^6(o_{\bf j}, 1)$ 
in $(CY^6(D_1), \partial^6)$. 
Since we have the equation 
$\partial^6 c^6({\bf j}, 1) 
= \mu a^6(\ell_{\bf j}, 1) + a^6(r_{\bf j}, 1) 
- a^6(\ell_{\bf j}, o_{\bf j}) a^6(o_{\bf j}, 1)$ 
in $(CY^6(D_1), \partial^6)$, 
we obtain 
$\partial^6 c^6({\bf j}, 0) = 0$ in $(CY^6(D_1), \partial^6)$. 
We notice equations 
$\partial^6 e^6({\bf j}, {\bf 0}) = c^6({\bf j}, 0)$, 
$\partial^6 c^6({\bf j}, 0) = 0$, and 
$\partial^6 e^6({\bf 0}, {\bf j}) = - \mu d^6(0, {\bf j})$, 
$\partial^6 d^6(0, {\bf j}) = 0$ 
in $(CY^6(D_1), \partial^6)$, 
where ${\bf j} \in \{ {\bf 1}, \cdots, {\bf m} \}$.

We eliminate pairs of generators 
$(e^6({\bf j}, {\bf 0}), c^6({\bf j}, 0))$ and 
$(e^6({\bf 0}, {\bf j}), d^6(0, {\bf j}))$ 
$({\bf j} \in \{ {\bf 1}, \cdots, {\bf m} \})$ 
by a sequence of destabilizations on $(CY^6(D_1), \partial^6)$, 
and we obtain $(CY^7(D_1), \partial^7)$. 
It is straightforward to see that 
the differential graded algebra $(CY^7(D_1), \partial^7)$ 
is isomorphic to $(CY(D_0), \partial)$. 
This shows that 
$(CY(D_1), \partial)$ is 
stably tame isomorphic to $(CY(D_0), \partial)$. 
\end{proof} 

We say that the series of tame isomorphisms and destabilizations 
performed in the proof of Proposition \ref{yoshikawa1-1} is 
a {\it destabilization along} ${\bf 0} \rightarrow 0$, 
from the crossing with label ${\bf 0}$ to the arc with label $0$, 
on $(CY(D_1), \partial)$.

\begin{proposition} \label{yoshikawa1-2} 
Suppose that a diagram $D_1$ of an oriented marked graph $G$ 
is obtained from a diagram $D_0$ of $G$ 
by applying one Yoshikawa move $\Omega_1'$, 
as illustrated in Figure \ref{yoshikawa1figure}. 
Then $(CY(D_1), \partial)$ is stably tame isomorphic to 
$(CY(D_0), \partial)$. 
\end{proposition} 

\begin{proof} 
We label arcs of $D_1$ involved in the Yoshikawa move 
by $0$ and $1$, 
and the crossing of $D_1$ 
created by the move 
by ${\bf 0}$. 
See Figure \ref{yoshikawa1figure}. 
Perform a destabilization along ${\bf 0} \rightarrow 0$ 
on $(CY(D_1), \partial)$, 
and we obtain a differential graded algebra 
that is isomorphic to $(CY(D_0), \partial)$. 
This shows that $(CY(D_1), \partial)$ is stably tame isomorphic to 
$(CY(D_0), \partial)$. 
\end{proof}

\section{Yoshikawa move $\Omega_2$}

\begin{figure}
\begin{center}
\includegraphics{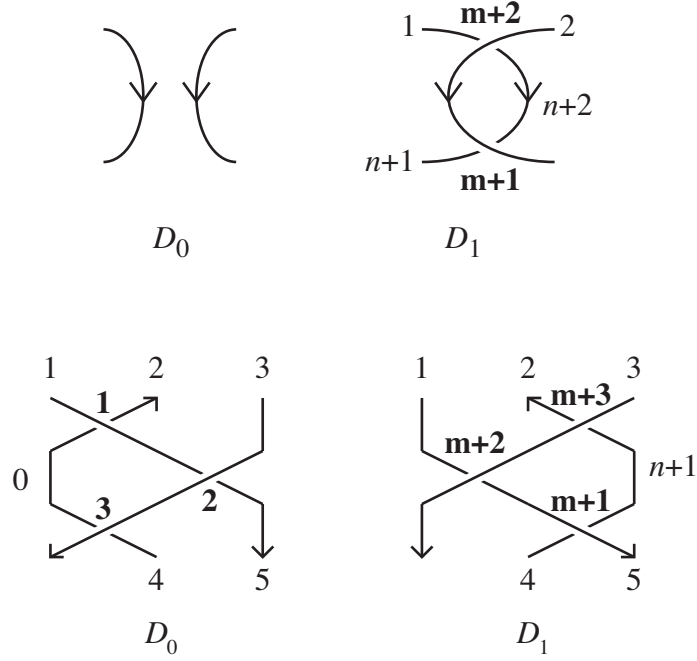}
\caption{Yoshikawa move $\Omega_2$ (upper pair) and 
Yoshikawa move $\Omega_3$ (lower pair)}
\label{yoshikawa23figure}
\end{center}
\end{figure}

\begin{proposition} \label{yoshikawa2} 
Suppose that a diagram $D_1$ of an oriented marked graph $G$ 
is obtained from a diagram $D_0$ of $G$ 
by applying one Yoshikawa move $\Omega_2$, 
as illustrated in Figure \ref{yoshikawa23figure}. 
Then 
$(CY(D_1), \partial)$ is 
stably tame isomorphic to $(CY(D_0), \partial)$. 
\end{proposition}

\begin{proof} 
We label arcs of $D_1$ involved in the Yoshikawa move 
by $1$, $n+1$, $n+2$ and $2$, and 
crossings of $D_1$ created by the move 
by ${\bf m+1}$ and ${\bf m+2}$. 
See Figure \ref{yoshikawa23figure}. 
Perform destabilizations on $(CY(D_1), \partial)$ 
along ${\bf m+1} \rightarrow n+1$, 
and 
${\bf m+2} \rightarrow n+2$ in this order, and 
we obtain a differential graded algebra 
that is isomorphic to $(CY(D_0), \partial)$. 
This shows that 
$(CY(D_1), \partial)$ is 
stably tame isomorphic to $(CY(D_0), \partial)$. 
\end{proof}

\section{Yoshikawa move $\Omega_3$}

\begin{proposition} \label{yoshikawa3} 
Suppose that a diagram $D_1$ of an oriented marked graph $G$ 
is obtained from a diagram $D_0$ of $G$ 
by applying one Yoshikawa move $\Omega_3$, 
as illustrated in Figure \ref{yoshikawa23figure}. 
Then 
$(CY(D_0), \partial)$ 
is stably tame isomorphic to $(CY(D_1), \partial)$. 
\end{proposition} 

\begin{proof} 
We label arcs of $D_0$ 
by $1, \cdots, n$, and 
crossings of $D_0$ 
by ${\bf 1}, \cdots, {\bf m}$. 
Let 
$0, 1, 2, 3, 4, 5$ (resp. ${\bf 1}, {\bf 2}, {\bf 3}$) 
denote 
labels on arcs (resp. crossings) of $D_0$ 
involved in the Yoshikawa move. 
See Figure \ref{yoshikawa23figure}. 
We label arcs (resp. crossings) of $D_1$ 
involved in the move 
by 
$1, 2, 3, 4, 5, n+1$ (resp. ${\bf m+1}, {\bf m+2}, {\bf m+3}$). 
We 
label arcs and crossings of $D_1$ so that 
an arc of $D_1$ with label $i$ $(i \in \{ 1, \cdots, n \})$ 
corresponds to the arc of $D_0$ with label $i$, and that 
a crossing of $D_1$ with label ${\bf j}$ 
$({\bf j} \in \{ {\bf 4}, {\bf 5}, \cdots, {\bf m} \})$ 
corresponds to the crossing of $D_0$ with label ${\bf j}$.

\subsection{Destabilization on $(CY(D_0), \partial)$} 

Perform a destabilization on $(CY(D_0), \partial)$ 
along ${\bf 1} \to 0$, and 
we obtain 
$(CY^0(D_0), \partial^0)$. 
The differential of $(CY^0(D_0), \partial^0)$ 
on generators 
is described as follows. \\ 
$\partial^0 c^0({\bf 2}, i) = 
\mu a^0(5, i) + a^0(1, i) - a^0(5, 3) a^0(3, i)$, \\ 
$\partial^0 d^0(i, {\bf 2}) = 
a^0(i, 5) + \mu a^0(i, 1) - a^0(i, 3) a^0(3, 5)$, \\ 
$\partial^0 c^0({\bf 3}, i) = 
\mu a^0(4, i) - \mu a^0(2, i) + a^0(2, 1) a^0(1, i) - a^0(4, 3) a^0(3, i)$, \\ 
$\partial^0 d^0(i, {\bf 3}) = 
a^0(i, 4) - a^0(i, 2) + a^0(i, 1) a^0(1, 2) - a^0(i, 3) a^0(3, 4)$, \\ 
$\partial^0 c^0({\bf j}, i) = \mu a^0(\ell_{\bf j}, i) + a^0(r_{\bf j}, i) 
- a^0(\ell_{\bf j}, o_{\bf j}) a^0(o_{\bf j}, i)$, \\ 
$\partial^0 d^0(i, {\bf j}) = a^0(i, \ell_{\bf j}) + \mu a^0(i, r_{\bf j}) 
- a^0(i, o_{\bf j}) a^0(o_{\bf j}, \ell_{\bf j})$, \\ 
$\partial^0 e^0({\bf j1}, {\bf j2}) = 
(c^0({\bf j1}, \ell_{\bf j2}) + \mu c^0({\bf j1}, r_{\bf j2}) 
- c^0({\bf j1}, o_{\bf j2}) a^0(o_{\bf j2}, \ell_{\bf j2}))$ \\ 
$- (\mu d^0(\ell_{\bf j1}, {\bf j2}) + d^0(r_{\bf j1}, {\bf j2}) 
- a^0(\ell_{\bf j1}, o_{\bf j1}) d^0(o_{\bf j1}, {\bf j2}))$, \\ 
$\partial^0 e^0({\bf j}, {\bf 2}) = 
(c^0({\bf j}, 5) + \mu c^0({\bf j}, 1) - c^0({\bf j}, 3) a^0(3, 5))$ 
$- (\mu d^0(\ell_{\bf j}, {\bf 2}) + d^0(r_{\bf j}, {\bf 2}) 
- a^0(\ell_{\bf j}, o_{\bf j}) d^0(o_{\bf j}, {\bf 2}))$, \\ 
$\partial^0 e^0({\bf 2}, {\bf j}) = 
(c^0({\bf 2}, \ell_{\bf j}) + \mu c^0({\bf 2}, r_{\bf j}) 
- c^0({\bf 2}, o_{\bf j}) a^0(o_{\bf j}, \ell_{\bf j}))$ 
$- (\mu d^0(5, {\bf j}) + d^0(1, {\bf j}) - a^0(5, 3) d^0(3, {\bf j}))$, \\ 
$\partial^0 e^0({\bf j}, {\bf 3}) = 
(c^0({\bf j}, 4) + c^0({\bf j}, 0) - c^0({\bf j}, 2) + c^0({\bf j}, 1) a^0(1, 2) - c^0({\bf j}, 3) a^0(3, 4))$ \\ 
$- (\mu d^0(\ell_{\bf j}, {\bf 3}) + d^0(r_{\bf j}, {\bf 3}) 
- a^0(\ell_{\bf j}, o_{\bf j}) d^0(o_{\bf j}, {\bf 3}))$, \\ 
$\partial^0 e^0({\bf 3}, {\bf j}) = 
(c^0({\bf 3}, \ell_{\bf j}) + \mu c^0({\bf 3}, r_{\bf j}) 
- c^0({\bf 3}, o_{\bf j}) a^0(o_{\bf j}, \ell_{\bf j}))$ \\ 
$- (\mu d^0(4, {\bf j}) + d^0(0, {\bf j}) - \mu d^0(2, {\bf j}) + a^0(2, 1) d^0(1, {\bf j}) - a^0(4, 3) d^0(3, {\bf j}))$, \\ 
$\partial^0 e^0({\bf 2}, {\bf 2}) = 
(c^0({\bf 2}, 5) + \mu c^0({\bf 2}, 1) - c^0({\bf 2}, 3) a^0(3, 5))$ 
$- (\mu d^0(5, {\bf 2}) + d^0(1, {\bf 2}) - a^0(5, 3) d^0(3, {\bf 2}))$, \\ 
$\partial^0 e^0({\bf 2}, {\bf 3}) = 
(c^0({\bf 2}, 4) - c^0({\bf 2}, 2) + c^0({\bf 2}, 1) a^0(1, 2) - c^0({\bf 2}, 3) a^0(3, 4))$ \\ 
$- (\mu d^0(5, {\bf 3}) + d^0(1, {\bf 3}) - a^0(5, 3) d^0(3, {\bf 3}))$, \\ 
$\partial^0 e^0({\bf 3}, {\bf 2}) = 
(c^0({\bf 3}, 5) + \mu c^0({\bf 3}, 1) - c^0({\bf 3}, 3) a^0(3, 5))$ \\ 
$- (\mu d^0(4, {\bf 2}) - \mu d^0(2, {\bf 2}) + a^0(2, 1) d^0(1, {\bf 2}) - a^0(4, 3) d^0(3, {\bf 2}))$, \\ 
$\partial^0 e^0({\bf 3}, {\bf 3}) = 
(c^0({\bf 3}, 4) - c^0({\bf 3}, 2) + c^0({\bf 3}, 1) a^0(1, 2) - c^0({\bf 3}, 3) a^0(3, 4))$ \\ 
$- (\mu d^0(4, {\bf 3}) - \mu d^0(2, {\bf 3}) + a^0(2, 1) d^0(1, {\bf 3}) - a^0(4, 3) d^0(3, {\bf 3}))$, \\ 
$\partial^0 f^0({\bf j}) = 
\mu c^0({\bf j}, r_{\bf j}) - \mu d^0(\ell_{\bf j}, {\bf j}) 
+ a^0(\ell_{\bf j}, o_{\bf j}) d^0(o_{\bf j}, {\bf j})$, \\ 
$\partial^0 f^0({\bf 2}) = 
\mu c^0({\bf 2}, 1) - \mu d^0(5, {\bf 2}) + a^0(5, 3) d^0(3, {\bf 2})$, \\ 
$\partial^0 f^0({\bf 3}) = 
- c^0({\bf 3}, 2) + c^0({\bf 3}, 1) a^0(1, 2) - \mu d^0(4, {\bf 3}) 
+ a^0(4, 3) d^0(3, {\bf 3})$, \\ 
where $i \in \{ 1, \cdots, n \}$, 
${\bf j}, {\bf j1}, {\bf j2} \in \{ {\bf 4}, {\bf 5}, \cdots, {\bf m} \}$.

\subsection{Destabilization on $(CY(D_1), \partial)$}

Perform a destabilization on $(CY(D_1), \partial)$ 
along ${\bf m+3} \to n+1$, and 
we obtain 
$(CY^1(D_1), \partial^1)$. 
The differential of $(CY^1(D_1), \partial^1)$ on generators 
is described as follows. \\ 
$\partial^1 c^1({\bf j}, i) = 
\mu a^1(\ell_{\bf j}, i) + a^1(r_{\bf j}, i) 
- a^1(\ell_{\bf j}, o_{\bf j}) a^1(o_{\bf j}, i)$, \\ 
$\partial^1 d^1(i, {\bf j}) = 
a^1(i, \ell_{\bf j}) + \mu a^1(i, r_{\bf j}) 
- a^1(i, o_{\bf j}) a^1(o_{\bf j}, \ell_{\bf j})$, \\ 
$\partial^1 c^1({\bf m+1}, i) = 
- a^1(2, i) + a^1(4, i) + \mu^{-1} a^1(2, 5) a^1(5, i) 
+ a^1(2, 3) (a^1(3, i) - \mu^{-1} a^1(3, 5) a^1(5, i))$, \\ 
$\partial^1 d^1(i, {\bf m+1}) = 
- \mu a^1(i, 2) + a^1(i, 3) a^1(3, 2) 
+ \mu a^1(i, 4) - a^1(i, 5) (- \mu a^1(5, 2) + a^1(5, 3) a^1(3, 2))$, \\ 
$\partial^1 c^1({\bf m+2}, i) = 
\mu a^1(5, i) + a^1(1, i) - a^1(5, 3) a^1(3, i)$, \\ 
$\partial^1 d^1(i, {\bf m+2}) = 
a^1(i, 5) + \mu a^1(i, 1) - a^1(i, 3) a^1(3, 5)$, \\ 
$\partial^1 e^1({\bf j1}, {\bf j2}) = 
(c^1({\bf j1}, \ell_{\bf j2}) + \mu c^1({\bf j1}, r_{\bf j2}) 
- c^1({\bf j1}, o_{\bf j2}) a^1(o_{\bf j2}, \ell_{\bf j2}))$ \\ 
$- (\mu d^1(\ell_{\bf j1}, {\bf j2}) + d^1(r_{\bf j1}, {\bf j2}) 
- a^1(\ell_{\bf j1}, o_{\bf j1}) d^1(o_{\bf j1}, {\bf j2}))$, \\ 
$\partial^1 e^1({\bf j}, {\bf m+1}) = 
(- \mu c^1({\bf j}, 2) + c^1({\bf j}, 3) a^1(3, 2) + \mu c^1({\bf j}, 4) 
- c^1({\bf j}, 5) (- \mu a^1(5, 2) + a^1(5, 3) a^1(3, 2)))$ \\ 
$- (\mu d^1(\ell_{\bf j}, {\bf m+1}) + d^1(r_{\bf j}, {\bf m+1}) 
- a^1(\ell_{\bf j}, o_{\bf j}) d^1(o_{\bf j}, {\bf m+1}))$, \\ 
$\partial^1 e^1({\bf m+1}, {\bf j}) = 
(c^1({\bf m+1}, \ell_{\bf j}) + \mu c^1({\bf m+1}, r_{\bf j}) 
- c^1({\bf m+1}, o_{\bf j}) a^1(o_{\bf j}, \ell_{\bf j}))$ \\ 
$- (- d^1(2, {\bf j}) + a^1(2, 3) d^1(3, {\bf j}) + d^1(4, {\bf j}) 
- (\mu^{-1} (- a^1(2, 5) + a^1(2, 3) a^1(3, 5))) d^1(5, {\bf j}))$, \\ 
$\partial^1 e^1({\bf j}, {\bf m+2}) = 
(c^1({\bf j}, 5) + \mu c^1({\bf j}, 1) - c^1({\bf j}, 3) a^1(3, 5))$ \\ 
$- (\mu d^1(\ell_{\bf j}, {\bf m+2}) + d^1(r_{\bf j}, {\bf m+2}) 
- a^1(\ell_{\bf j}, o_{\bf j}) d^1(o_{\bf j}, {\bf m+2}))$, \\ 
$\partial^1 e^1({\bf m+2}, {\bf j}) = 
(c^1({\bf m+2}, \ell_{\bf j}) + \mu c^1({\bf m+2}, r_{\bf j}) 
- c^1({\bf m+2}, o_{\bf j}) a^1(o_{\bf j}, \ell_{\bf j}))$ \\ 
$- (\mu d^1(5, {\bf j}) + d^1(1, {\bf j}) - a^1(5, 3) d^1(3, {\bf j}))$, \\ 
$\partial^1 e^1({\bf m+1}, {\bf m+1}) = 
(- \mu c^1({\bf m+1}, 2) + c^1({\bf m+1}, 3) a^1(3, 2)$ \\ 
$+ \mu c^1({\bf m+1}, 4) 
- c^1({\bf m+1}, 5) (- \mu a^1(5, 2) + a^1(5, 3) a^1(3, 2)))$ \\ 
$- (- d^1(2, {\bf m+1}) + a^1(2, 3) d^1(3, {\bf m+1}) + d^1(4, {\bf m+1})$ \\ 
$- (\mu^{-1} (- a^1(2, 5) + a^1(2, 3) a^1(3, 5))) d^1(5, {\bf m+1}))$, \\ 
$\partial^1 e^1({\bf m+1}, {\bf m+2}) = 
(c^1({\bf m+1}, 5) + \mu c^1({\bf m+1}, 1) - c^1({\bf m+1}, 3) a^1(3, 5))$ \\ 
$- (- (d^1(2, {\bf m+2}) - a^1(2, 3) d^1(3, {\bf m+2})) + d^1(4, {\bf m+2})$ \\ 
$- (\mu^{-1} (- a^1(2, 5) + a^1(2, 3) a^1(3, 5))) d^1(5, {\bf m+2}))$, \\ 
$\partial^1 e^1({\bf m+2}, {\bf m+1}) = 
(- \mu c^1({\bf m+2}, 2) + c^1({\bf m+2}, 3) a^1(3, 2) + \mu c^1({\bf m+2}, 4)$ \\ 
$- c^1({\bf m+2}, 5) (- \mu a^1(5, 2) + a^1(5, 3) a^1(3, 2)))$ \\ 
$- (\mu d^1(5, {\bf m+1}) + d^1(1, {\bf m+1}) - a^1(5, 3) d^1(3, {\bf m+1}))$, \\ 
$\partial^1 e^1({\bf m+2}, {\bf m+2}) = 
(c^1({\bf m+2}, 5) + \mu c^1({\bf m+2}, 1) - c^1({\bf m+2}, 3) a^1(3, 5))$ \\ 
$- (\mu d^1(5, {\bf m+2}) + d^1(1, {\bf m+2}) - a^1(5, 3) d^1(3, {\bf m+2}))$, \\ 
$\partial^1 f^1({\bf j}) = 
\mu c^1({\bf j}, r_{\bf j}) - \mu d^1(\ell_{\bf j}, {\bf j}) 
+ a^1(\ell_{\bf j}, o_{\bf j}) d^1(o_{\bf j}, {\bf j})$ \\ 
$\partial^1 f^1({\bf m+1}) = 
\mu c^1({\bf m+1}, 4) - \mu (- \mu^{-1} d^1(2, {\bf m+1}) 
+ \mu^{-1} a^1(2, 3) d^1(3, {\bf m+1}))$ \\ 
$+ (\mu^{-1} (- a^1(2, 5) + a^1(2, 3) a^1(3, 5))) d^1(5, {\bf m+1})$, \\ 
$\partial^1 f^1({\bf m+2}) = 
\mu c^1({\bf m+2}, 1) - \mu d^1(5, {\bf m+2}) + a^1(5, 3) d^1(3, {\bf m+2})$, \\ 
where $i \in \{ 1, \cdots, n \}$, 
${\bf j}, {\bf j1}, {\bf j2} \in \{ {\bf 4}, {\bf 5}, \cdots, {\bf m} \}$.

\subsection{Tame isomorphism 
between $(CY^0(D_0), \partial^0)$ and $(CY^1(D_1), \partial^1)$} 

In this subsection, 
we construct a tame isomorphism 
from $(CY^0(D_0), \partial^0)$ to $(CY^1(D_1), \partial^1)$.

First we define a tame isomorphism 
$\varphi^0 \colon (CY^0(D_0), \partial^0) \to (CY^9(D_0), \partial^9)$ 
by \\
$\varphi^0 (f^0({\bf 3})) = 
- f^9({\bf 3}) 
+ e^9({\bf 3}, {\bf 3}) 
- a^9(2, 1) e^9({\bf 2}, {\bf 3}) 
+ \mu^{-1} a^9(2, 1) f^9({\bf 2}) a^9(1, 2)$ \\ 
$+ \mu^{-1} d^9(2, {\bf 2}) 
(a^9(5, 3) d^9(3, {\bf 3}) - c^9({\bf 2}, 2) - \mu d^9(5, {\bf 3}) 
- \mu^{-1} a^9(5, 3) d^9(3, {\bf 2}) a^9(1, 2) + d^9(5, {\bf 2}) a^9(1, 2))$, \\ 
$\varphi^0 (a^0(i, j)) = a^9(i, j)$, 
$\varphi^0 (c^0({\bf k}, i)) = c^9({\bf k}, i)$, 
$\varphi^0 (d^0(i, {\bf k})) = d^9(i, {\bf k})$, \\ 
$\varphi^0 (e^0({\bf k1}, {\bf k2})) = e^9({\bf k1}, {\bf k2})$, 
$\varphi^0 (f^0({\bf j})) = f^9({\bf j})$, \\ 
where 
$i \neq j \in \{ 1, \cdots, n \}$, 
${\bf j} \in \{ {\bf 2} \} \cup \{ {\bf 4}, \cdots, {\bf m} \}$, 
${\bf k}, {\bf k1}, {\bf k2} \in \{ {\bf 2}, \cdots, {\bf m} \}$.

Next we define a tame isomorphism 
$\varphi^9 \colon (CY^9(D_0), \partial^9) \to (CY^8(D_0), \partial^8)$ 
by \\ 
$\varphi^9 (e^9({\bf 3}, {\bf 3})) 
= e^8({\bf 3}, {\bf 3}) 
- a^8(2, 1) e^8({\bf 2}, {\bf 3}) 
- \mu^{-1} e^8({\bf 3}, {\bf 2}) a^8(1, 2)$ 
$+ \mu^{-1} a^8(2, 1) e^8({\bf 2}, {\bf 2}) a^8(1, 2)$ \\ 
$- (\mu^{-1} c^8({\bf 3}, 5) + d^8(2, {\bf 2}) 
- \mu^{-1} c^8({\bf 3}, 3) a^8(3, 5)) 
c^8({\bf 2}, 2)$ \\ 
$- \mu^{-1} d^8(2, {\bf 2}) 
(\mu d^8(5, {\bf 3}) + c^8({\bf 2}, 2) - a^8(5, 3) d^8(3, {\bf 3}))$ \\ 
$+ \mu^{-1} a^8(2, 1) 
(c^8({\bf 2}, 5) - c^8({\bf 2}, 3) a^8(3, 5)) c^8({\bf 2}, 2)$ \\ 
$+ \mu^{-1} d^8(2, {\bf 2}) 
(d^8(5, {\bf 2}) - \mu^{-1} a^8(5, 3) d^8(3, {\bf 2})) a^8(1, 2)$, \\ 
$\varphi^9 (a^9(i, j)) = a^8(i, j)$, 
$\varphi^9 (c^9({\bf k}, i)) = c^8({\bf k}, i)$, 
$\varphi^9 (d^9(i, {\bf k})) = d^8(i, {\bf k})$, \\ 
$\varphi^9 (e^9({\bf k1}, {\bf k2})) = e^8({\bf k1}, {\bf k2})$, 
$\varphi^9 (f^9({\bf k})) = f^8({\bf k})$, \\ 
where 
$i \neq j \in \{ 1, \cdots, n \}$, 
${\bf k}, {\bf k1}, {\bf k2} \in \{ {\bf 2}, \cdots, {\bf m} \}$, 
and $({\bf k1}, {\bf k2}) \neq ({\bf 3}, {\bf 3})$.

Next we define a tame isomorphism 
$\varphi^8 \colon (CY^8(D_0), \partial^8) \to (CY^7(D_0), \partial^7)$ 
by \\ 
$\varphi^8 (e^8({\bf j}, {\bf 3})) 
= \mu e^7({\bf j}, {\bf 3}) 
- e^7({\bf j}, {\bf 2}) a^7(1, 2) 
- c^7({\bf j}, 5) c^7({\bf 2}, 2)$ 
$+ c^7({\bf j}, 3) (d^7(3, {\bf 3}) - \mu^{-1} d^7(3, {\bf 2}) a^7(1, 2))$, \\ 
$\varphi^8 (e^8({\bf 3}, {\bf j})) 
= \mu^{-1} (e^7({\bf 3}, {\bf j}) 
- a^7(2, 1) e^7({\bf 2}, {\bf j}) 
- d^7(2, {\bf 2}) d^7(5, {\bf j})$ 
$- (a^7(2, 1) c^7({\bf 2}, 3) - c^7({\bf 3}, 3)) d^7(3, {\bf j}))$, \\ 
$\varphi^8 (a^8(i, j)) = a^7(i, j)$, 
$\varphi^8 (c^8({\bf k}, i)) = c^7({\bf k}, i)$, 
$\varphi^8 (d^8(i, {\bf k})) = d^7(i, {\bf k})$, \\ 
$\varphi^8 (e^8({\bf j1}, {\bf j2})) = e^7({\bf j1}, {\bf j2})$, 
$\varphi^8 (e^8({\bf 3}, {\bf 3})) = e^7({\bf 3}, {\bf 3})$, 
$\varphi^8 (f^8({\bf k})) = f^7({\bf k})$, \\ 
where 
$i \neq j \in \{ 1, \cdots, n \}$, 
${\bf k} \in \{ {\bf 2}, \cdots, {\bf m} \}$, 
${\bf j}, {\bf j1}, {\bf j2} \in \{ {\bf 2} \} \cup \{ {\bf 4}, \cdots, {\bf m} \}$. 

Next we define a tame isomorphism 
$\varphi^7 \colon (CY^7(D_0), \partial^7) \to (CY^6(D_0), \partial^6)$ 
by \\ 
$\varphi^7 (c^7({\bf 3}, 3)) 
= - c^6({\bf 3}, 3) 
+ \mu^{-1} d^6(2, {\bf 2}) a^6(5, 3) 
+ a^6(2, 1) c^6({\bf 2}, 3)$, \\ 
$\varphi^7 (d^7(3, {\bf 3})) 
= - d^6(3, {\bf 3})
+ a^6(3, 5) c^6({\bf 2}, 2) 
+ \mu^{-1} d^6(3, {\bf 2}) a^6(1, 2)$, \\ 
$\varphi^6 (a^6(i, j)) = a^5(i, j)$, 
$\varphi^6 (c^6({\bf k}, i)) = c^5({\bf k}, i)$, 
$\varphi^6 (d^6(i, {\bf k})) = d^5(i, {\bf k})$, \\ 
$\varphi^6 (e^6({\bf k1}, {\bf k2})) = e^5({\bf k1}, {\bf k2})$, 
$\varphi^6 (f^6({\bf k})) = f^5({\bf k})$, \\ 
where 
$i \neq j \in \{ 1, \cdots, n \}$, 
${\bf k}, {\bf k1}, {\bf k2} \in \{ {\bf 2}, \cdots, {\bf m} \}$, 
$({\bf k}, i) \neq ({\bf 3}, 3)$.

Next we define a tame isomorphism 
$\varphi^6 \colon (CY^6(D_0), \partial^6) \to (CY^5(D_0), \partial^5)$ 
by \\ 
$\varphi^6 (c^6({\bf 3}, k)) 
= \mu^{-1} (c^5({\bf 3}, k) 
+ d^5(2, {\bf 2}) a^5(5, k) 
- a^5(2, 1) c^5({\bf 2}, k) 
+ a^5(2, 1) c^5({\bf 2}, 3) a^5(3, k) 
- c^5({\bf 3}, 3) a^5(3, k))$, \\ 
$\varphi^6 (d^6(k, {\bf 3})) 
= \mu d^5(k, {\bf 3}) 
+ a^5(k, 5) c^5({\bf 2}, 2) 
- a^5(k, 3) d^5(3, {\bf 3}) 
+ \mu^{-1} a^5(k, 3) d^5(3, {\bf 2}) a^5(1, 2) 
- d^5(k, {\bf 2}) a^5(1, 2)$, \\ 
$\varphi^6 (a^6(i, j)) = a^5(i, j)$, 
$\varphi^6 (c^6({\bf j}, i)) = c^5({\bf j}, i)$, 
$\varphi^6 (d^6(i, {\bf j})) = d^5(i, {\bf j})$, \\ 
$\varphi^6 (c^6({\bf 3}, 3)) = c^5({\bf 3}, 3)$, 
$\varphi^6 (d^6(3, {\bf 3})) = d^5(3, {\bf 3})$, 
$\varphi^6 (e^6({\bf k1}, {\bf k2})) = e^5({\bf k1}, {\bf k2})$, 
$\varphi^6 (f^6({\bf k})) = f^5({\bf k})$, \\ 
where 
$i \neq j \in \{ 1, \cdots, n \}$, 
$k \in \{ 1, 2 \} \cup \{ 4, 5, \cdots, n \}$, 
${\bf j} \in \{ {\bf 2} \} \cup \{ {\bf 4}, \cdots, {\bf m} \}$, 
${\bf k}, {\bf k1}, {\bf k2} \in \{ {\bf 2}, \cdots, {\bf m} \}$.

Finally we define an isomorphism 
$\varphi^5 \colon (CY^5(D_0), \partial^5) \to (CY^4(D_0), \partial^4)$ 
by \\ 
$\varphi^5 (a^5(i, j)) = a^4(i, j)$, 
$\varphi^5 (c^5({\bf k}, i)) = c^4({\bf k}, i)$, 
$\varphi^5 (d^5(i, {\bf k})) = d^5(i, {\bf k})$, \\ 
$\varphi^5 (c^5({\bf 2}, i)) = c^4({\bf m+2}, i)$, 
$\varphi^5 (d^5(i, {\bf 2})) = d^4(i, {\bf m+2})$, \\ 
$\varphi^5 (c^5({\bf 3}, i)) = c^4({\bf m+1}, i)$, 
$\varphi^5 (d^5(i, {\bf 3})) = d^4(i, {\bf m+1})$, \\ 
$\varphi^5 (e^5({\bf k}, {\bf j})) = e^4({\bf k}, {\bf j})$, 
$\varphi^5 (f^5({\bf k})) = f^4({\bf k})$, \\ 
$\varphi^5 (e^5({\bf 2}, {\bf 2})) = e^4({\bf m+2}, {\bf m+2})$, 
$\varphi^5 (e^5({\bf 2}, {\bf 3})) = e^4({\bf m+2}, {\bf m+1})$, \\ 
$\varphi^5 (e^5({\bf 3}, {\bf 2})) = e^4({\bf m+1}, {\bf m+2})$, 
$\varphi^5 (e^5({\bf 3}, {\bf 3})) = e^4({\bf m+1}, {\bf m+1})$, \\ 
$\varphi^5 (f^5({\bf 2})) = f^4({\bf m+2})$, 
$\varphi^5 (f^5({\bf 3})) = f^4({\bf m+1})$, \\ 
where 
$i \neq j \in \{ 1, \cdots, n \}$, 
${\bf j}, {\bf k} \in \{ {\bf 4}, \cdots, {\bf m} \}$.

It is straightforward to see that 
$(CY^4(D_0), \partial^4)$ is isomorphic to $(CY^1(D_1), \partial^1)$. 
This shows that 
$(CY^0(D_0), \partial^0)$ 
is tame isomorphic to $(CY^1(D_1), \partial^1)$, and that 
$(CY(D_0), \partial)$ 
is stably tame isomorphic to $(CY(D_1), \partial)$. 
\end{proof}

\section{Yoshikawa move $\Omega_4$}

\begin{figure}
\begin{center}
\includegraphics{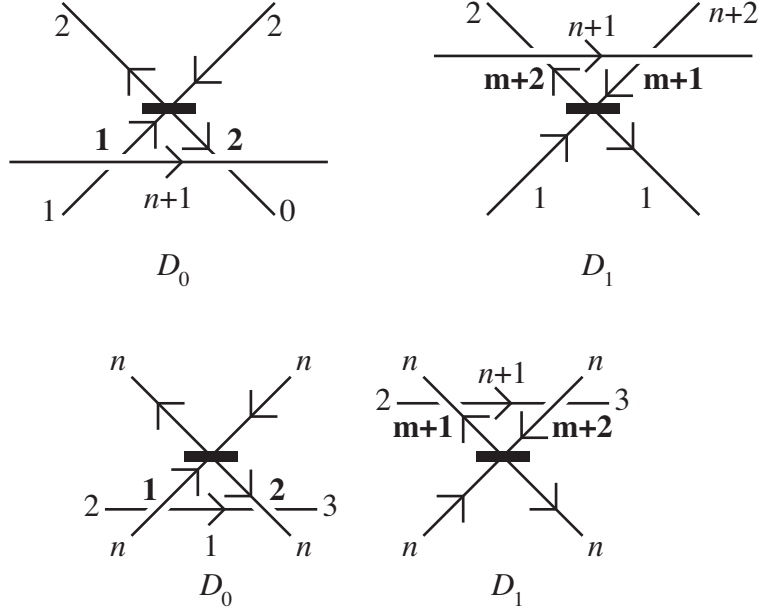}
\caption{Yoshikawa move $\Omega_4$ (upper pair) and 
Yoshikawa move $\Omega_4'$ (lower pair)}
\label{yoshikawa4figure}
\end{center}
\end{figure}

\begin{proposition} \label{yoshikawa4-1} 
Suppose that a diagram $D_1$ of an oriented marked graph $G$ 
is obtained from a diagram $D_0$ of $G$ 
by applying one Yoshikawa move $\Omega_4$, 
as illustrated in Figure \ref{yoshikawa4figure}. 
Then $(CY(D_0), \partial)$ is stably tame isomorphic to 
$(CY(D_1), \partial)$. 
\end{proposition}

\begin{proof} 
Let ${\bf 1}, {\bf 2}$ denote labels on crossings of $D_0$ 
involved in the Yoshikawa move. 
See Figure \ref{yoshikawa4figure}. 
Let $n+1$ denote a label on 
the over-arc of $D_0$ 
at the crossing with label ${\bf 1}$. 
The arc of $D_0$ 
that contains the marked vertex involved in the move 
is labeled by $2$. 
Let $1$ and $2$ (resp. $0$ and $2$) denote labels 
on under-arcs of $D_0$ 
at the crossing with label ${\bf 1}$ (resp. ${\bf 2}$). 
Perform destabilizations on $(CY(D_0), \partial)$ along ${\bf 2} \to 0$, 
and ${\bf 1} \to 2$ in this order, and 
we obtain $(CY^0(D_0), \partial^0)$.

Let ${\bf m+1}, {\bf m+2}$ 
denote labels on crossings of $D_1$ 
involved in the move. 
Let $n+1$ denote a label on 
the over-arc of $D_1$ 
at the crossing with label ${\bf m+1}$. 
The arc of $D_1$ 
that contains the marked vertex involved in the move 
is labeled by $1$. 
Let $n+2$ and $1$ (resp. $2$ and $1$) denote labels 
on under-arcs of $D_1$ 
at the crossing with label ${\bf m+1}$ (resp. ${\bf m+2}$). 
Perform destabilizations on $(CY(D_1), \partial)$ along ${\bf m+1} \to n+2$, 
and ${\bf m+2} \to 2$ in this order, and 
we obtain $(CY^1(D_1), \partial^1)$.

It is straightforward to see that 
$(CY^0(D_0), \partial^0)$ is 
isomorphic to 
$(CY^1(D_1), \partial^1)$. 
This shows that $(CY(D_0), \partial)$ is 
stably tame isomorphic to $(CY(D_1), \partial)$. 
\end{proof}

\section{Yoshikawa move $\Omega_4'$}

\begin{proposition} \label{yoshikawa4-2} 
Suppose that a diagram $D_1$ of an oriented marked graph $G$ 
is obtained from a diagram $D_0$ of $G$ 
by applying one Yoshikawa move $\Omega_4'$, 
as illustrated in Figure \ref{yoshikawa4figure}. 
Then 
$(CY(D_0), \partial)$ is isomorphic to 
$(CY(D_1), \partial)$. 
\end{proposition}

\begin{proof} 
Let ${\bf 1}, {\bf 2}$ denote labels on crossings of $D_0$ 
involved in the Yoshikawa move. 
See Figure \ref{yoshikawa4figure}. 
The arc of $D_0$ that contains the marked vertex involved in the move 
is labeled by $n$. 
Let $1$ and $2$ (resp. $1$ and $3$) 
denote labels on under-arcs of $D_0$ 
at the crossing with label ${\bf 1}$ (resp. ${\bf 2}$). 

Let ${\bf m+1}, {\bf m+2}$ denote labels on crossings of $D_1$ 
involved in the move. 
The arc of $D_1$ 
that contains the marked vertex involved in the move 
is labeled by $n$. 
Let $n+1$ and $2$ (resp. $n+1$ and $3$) 
denote labels on under-arcs of $D_1$ 
at the crossing with label ${\bf m+1}$ (resp. ${\bf m+2}$).

We define an isomorphism 
$\phi \colon (CY(D_0), \partial) \to (CY(D_1), \partial)$ 
by \\ 
$\phi (a(k, 1)) = a(k, n+1)$, 
$\phi (a(1, k)) = a(n+1, k)$, 
$\phi (a(j, i)) = a(j, i)$, \\ 
$\phi (c({\bf x}, i)) = c({\bf x}, i)$, 
$\phi (d(i, {\bf x})) = d(i, {\bf x})$, 
$\phi (c({\bf x}, 1)) = c({\bf x}, n+1)$, 
$\phi (d(1, {\bf x})) = d(n+1, {\bf x})$, \\ 
$\phi (c({\bf 1}, k)) = c({\bf m+1}, k)$, 
$\phi (c({\bf 1}, 1)) = c({\bf m+1}, n+1)$, \\ 
$\phi (d(k, {\bf 1})) = d(k, {\bf m+1})$, 
$\phi (d(1, {\bf 1})) = d(n+1, {\bf m+1})$, \\ 
$\phi (c({\bf 2}, k)) = c({\bf m+2}, k)$, 
$\phi (c({\bf 2}, 1)) = c({\bf m+2}, n+1)$, \\ 
$\phi (d(k, {\bf 2})) = d(k, {\bf m+2})$, 
$\phi (d(1, {\bf 2})) = d(n+1, {\bf m+2})$, \\ 
$\phi (e({\bf x}, {\bf 1})) = e({\bf x}, {\bf m+1})$, 
$\phi (e({\bf 1}, {\bf x})) = e({\bf m+1}, {\bf x})$, \\ 
$\phi (e({\bf x}, {\bf 2})) = e({\bf x}, {\bf m+2})$, 
$\phi (e({\bf 2}, {\bf x})) = e({\bf m+2}, {\bf x})$, \\ 
$\phi (e({\bf 1}, {\bf 1})) = e({\bf m+1}, {\bf m+1})$, 
$\phi (e({\bf 1}, {\bf 2})) = e({\bf m+1}, {\bf m+2})$, \\ 
$\phi (e({\bf 2}, {\bf 1})) = e({\bf m+2}, {\bf m+1})$, 
$\phi (e({\bf 2}, {\bf 2})) = e({\bf m+2}, {\bf m+2})$, \\ 
$\phi (f({\bf x})) = f({\bf x})$, 
$\phi (f({\bf 1})) = f({\bf m+1})$, 
$\phi (f({\bf 2})) = f({\bf m+2})$ \\ 
for $k \in \{ 2, 3, \cdots, n \}$, 
$i \neq j \in \{ 1, \cdots, n \}$, 
${\bf x} \in \{ {\bf 3}, {\bf 4}, \cdots, {\bf m} \}$. 
This 
shows that 
$(CY(D_0), \partial)$ is isomorphic to $(CY(D_1), \partial)$. 
\end{proof}

\section{Yoshikawa move $\Omega_5$}

\begin{figure}
\begin{center}
\includegraphics{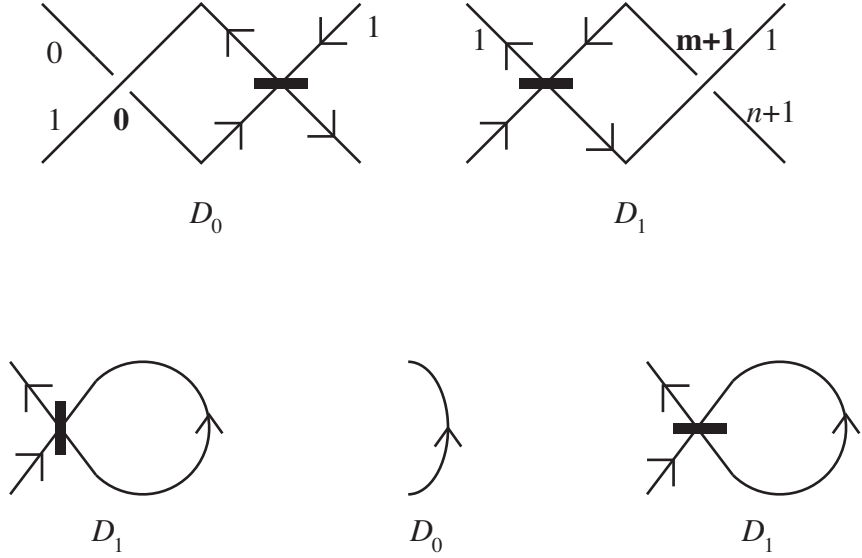}
\caption{Yoshikawa move $\Omega_5$ (upper pair) and 
Yoshikawa move $\Omega_6$ (lower left pair) and 
$\Omega_6'$ (lower right pair)}
\label{yoshikawa56figure}
\end{center}
\end{figure}

\begin{proposition} \label{yoshikawa5} 
Suppose that a diagram $D_1$ of an oriented marked graph $G$ 
is obtained from a diagram $D_0$ of $G$ 
by applying one Yoshikawa move $\Omega_5$, 
as illustrated in Figure \ref{yoshikawa56figure}. 
Then 
$(CY(D_0), \partial)$ 
is stably tame isomorphic to $(CY(D_1), \partial)$. 
\end{proposition}

\begin{proof} 
Let ${\bf 0}$ denote a label on the crossing of $D_0$ 
involved in the Yoshikawa move. 
See Figure \ref{yoshikawa56figure}. 
Let $1$ denote a label on the arc of $D_0$ 
that contains the marked vertex involved in the move. 
Under-arcs of $D_0$ 
at the crossing with label ${\bf 0}$ 
are labeled by $0$ and $1$. 
Perform a destabilization along ${\bf 0} \to 0$ on $(CY(D_0), \partial)$, and 
we obtain $(CY^0(D_0), \partial^0)$. 

Let ${\bf m+1}$ denote a label on the crossing of $D_1$ 
involved in the move. 
Let $1$ denote a label on the arc of $D_1$ 
that contains the marked vertex involved in the move. 
Under-arcs of $D_1$ 
at the crossing with label ${\bf m+1}$ 
are labeled by $1$ and $n+1$. 
Perform a destabilization along ${\bf m+1} \to n+1$ on $(CY(D_1), \partial)$, and 
we obtain $(CY^1(D_1), \partial^1)$. 

It is straightforward to see that 
$(CY^0(D_0), \partial^0)$ is isomorphic to $(CY^1(D_1), \partial^1)$. 
This shows that $(CY(D_0), \partial)$ 
is stably tame isomorphic to $(CY(D_1), \partial)$. 
\end{proof}

\section{Yoshikawa moves $\Omega_6$, $\Omega_6'$ and $\Omega_7$}

We notice that 
one arc is involved in the Yoshikawa move $\Omega_6$, $\Omega_6'$ 
and $\Omega_7$. 
It is straightforward to see the following propositions.

\begin{proposition} \label{yoshikawa6-1} 
Suppose that a diagram $D_1$ of an oriented marked graph $G$ 
is obtained from a diagram $D_0$ of $G$ 
by applying one Yoshikawa move 
$\Omega_6$ or $\Omega_6'$, 
as illustrated in Figure \ref{yoshikawa56figure}. 
Then 
$(CY(D_0), \partial)$ 
is isomorphic to $(CY(D_1), \partial)$. 
\end{proposition}

\begin{figure}
\begin{center}
\includegraphics{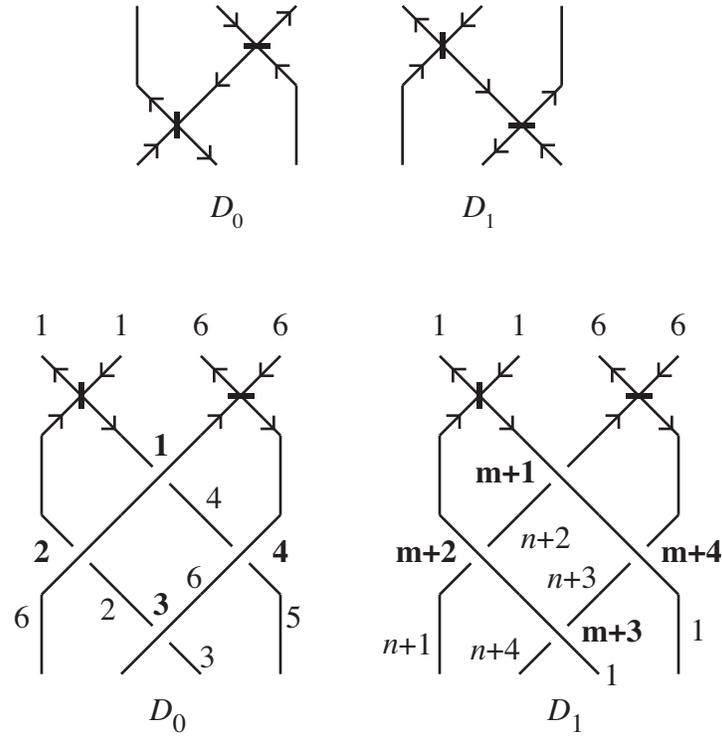}
\caption{Yoshikawa move $\Omega_7$ (upper pair) and 
Yoshikawa move $\Omega_8$ (lower pair)}
\label{yoshikawa78figure}
\end{center}
\end{figure}

\begin{proposition} \label{yoshikawa7} 
Suppose that a diagram $D_1$ of an oriented marked graph $G$ 
is obtained from a diagram $D_0$ of $G$ 
by applying one Yoshikawa move $\Omega_7$, 
as illustrated in Figure \ref{yoshikawa78figure}. 
Then 
$(CY(D_0), \partial)$ 
is isomorphic to $(CY(D_1), \partial)$. 
\end{proposition}

\section{Yoshikawa move $\Omega_8$}

\begin{proposition} \label{yoshikawa8} 
Suppose that a diagram $D_1$ of an oriented marked graph $G$ 
is obtained from a diagram $D_0$ of $G$ 
by applying one Yoshikawa move $\Omega_8$, 
as illustrated in Figure \ref{yoshikawa78figure}. 
Then 
$(CY(D_0), \partial)$ 
is stably tame isomorphic to $(CY(D_1), \partial)$. 
\end{proposition}

\begin{proof} 
Let $1, 2, 3, 4, 5, 6$ denote 
labels on arcs of $D_0$ involved in the Yoshikawa move. 
See Figure \ref{yoshikawa78figure}. 
Let ${\bf 1}, {\bf 2}, {\bf 3}, {\bf 4}$ 
denote labels on crossings of $D_0$ involved in the move. 
Perform a sequence of destabilizations on $(CY(D_0), \partial)$ 
along ${\bf 1} \rightarrow 4$, 
${\bf 4} \rightarrow 5$, 
${\bf 2} \rightarrow 2$, and 
${\bf 3} \rightarrow 3$ 
in this order, and 
we obtain 
$(CY^0(D_0), \partial^0)$.

Let $1, n+1, n+2, n+3, n+4, 6$ denote 
labels on arcs of $D_1$ involved in the move. 
Let ${\bf m+1}, {\bf m+2}, {\bf m+3}, {\bf m+4}$ 
denote labels on crossings of $D_1$ involved in the move. 
Perform a sequence of destabilizations 
on $(CY(D_1), \partial)$ 
along ${\bf m+1} \rightarrow n+2$, 
${\bf m+2} \rightarrow n+1$, 
${\bf m+4} \rightarrow n+3$, and 
${\bf m+3} \rightarrow n+4$ 
in this order, and 
we obtain 
$(CY^1(D_1), \partial^1)$.

It is straightforward to see that 
$(CY^0(D_0), \partial^0)$ 
is isomorphic to $(CY^1(D_1), \partial^1)$. 
This shows that $(CY(D_0), \partial)$ is 
stably tame isomorphic to $(CY(D_1), \partial)$. 
\end{proof}

\section{Proof of Theorem \ref{main-theorem}} 

We may assume by Theorem \ref{kearton-kurlin-swenton} that 
$D_1$ is obtained from $D_0$ by a finite sequence of Yoshikawa moves 
$\Omega_1$, $\Omega_1'$, $\Omega_2$, $\Omega_3$, $\Omega_4$, 
$\Omega_4'$, 
$\Omega_5$, $\Omega_6$, $\Omega_6'$, $\Omega_7$, $\Omega_8$. 
Propositions \ref{yoshikawa1-1}, \ref{yoshikawa1-2}, 
\ref{yoshikawa2}, \ref{yoshikawa3}, \ref{yoshikawa4-1}, \ref{yoshikawa4-2}, 
\ref{yoshikawa5}, \ref{yoshikawa6-1}, \ref{yoshikawa7}, \ref{yoshikawa8} 
complete the proof of Theorem \ref{main-theorem}.

\section{Examples} 

Let $T(2, 3)$ denote a torus knot of type $(2, 3)$ in ${\mathbb{R}}^3$. 
Let $T^0(2, 3)$ 
denote a 2-sphere in ${\mathbb{R}}^4$ 
that is obtained from $T(2, 3)$ 
by the spinning construction, 
introduced by Artin \cite{artin}. 
The upper diagram in Figure \ref{example} illustrates a diagram $D^0(2, 3)$ of 
an oriented marked graph representing $T^0(2, 3)$, 
denoted by $8_1$ in Yoshikawa's table \cite{yoshikawa}. 

\begin{figure}
\begin{center}
\includegraphics{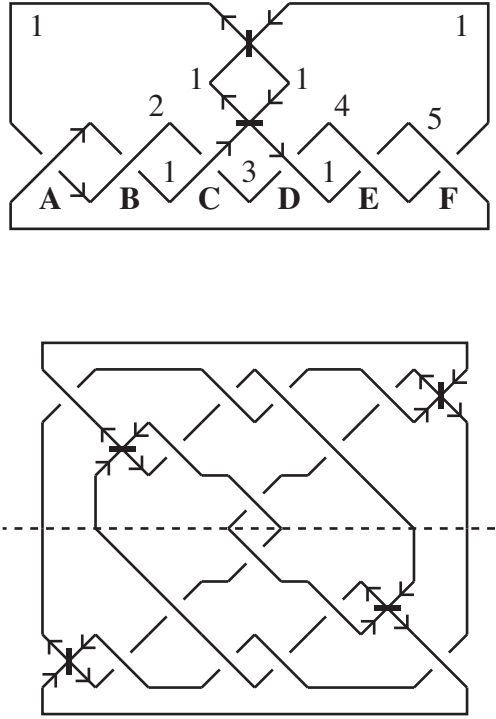}
\caption{oriented marked graph diagram $D^0(2, 3)$ representing 
$T^0(2, 3)$ (upper) and 
oriented marked graph diagram $D^2(2, 3)$ representing 
$T^2(2, 3)$ (lower)}
\label{example}
\end{center}
\end{figure}

\begin{theorem} \label{calculation-spun-trefoil} 
The 0-dimensional Yoshikawa homology of $T^0(2, 3)$, 
$HY_0(T^0(2, 3))$, is isomorphic to 
${\mathbb{Z}}[\mu, \mu^{-1}] \langle x, y \rangle / 
(\mu (1 + \mu) + x - y x, 
(1 + \mu) + \mu y - y x, 
\mu (1 + \mu) + \mu^{-1} x - \mu^{-1} x y x)$.
\end{theorem}

\begin{proof} 
Arcs of $D^0(2, 3)$ are labeled by $1, 2, 3, 4, 5$, and 
crossings of $D^0(2, 3)$ are labeled 
by ${\bf A}, {\bf B}, {\bf C}, {\bf D}, {\bf E}, {\bf F}$, 
as illustrated in Figure \ref{example}. 
The diagram $D^0(2, 3)$ shows that 
the differential of $(CY_1(D^0(2, 3)), \partial)$ on generators is 
described as follows. 

\noindent 
$\partial c({\bf A}, i) = 
\mu a(1, i) 
+ a(2, i) 
- a(1, 5) 
a(5, i)$, 
$\partial d(i, {\bf A}) = 
a(i, 1) 
+ \mu a(1, 2) 
- a(i, 5) 
a(5, 1)$,

\noindent 
$\partial c({\bf B}, i) = 
\mu a(5, i) 
+ a(1, i) 
- a(5, 2) 
a(2, i)$, 
$\partial d(i, {\bf B}) = 
a(i, 5) 
+ \mu a(i, 1) 
- a(i, 2) 
a(2, 5)$, 

\noindent 
$\partial c({\bf C}, i) = 
\mu a(2, i) 
+ a(3, i) 
- a(2, 1) 
a(1, i)$, 
$\partial d(i, {\bf C}) = 
a(i, 2) 
+ \mu a(1, 3) 
- a(i, 1) 
a(1, 2)$, 

\noindent 
$\partial c({\bf D}, i) = 
\mu a(4, i) 
+ a(3, i) 
- a(4, 1) 
a(1, i)$, 
$\partial d(i, {\bf D}) = 
a(i, 4) 
+ \mu a(i, 3) 
- a(i, 1) 
a(1, 4)$,

\noindent 
$\partial c({\bf E}, i) = 
\mu a(5, i) 
+ a(1, i) 
- a(5, 4) 
a(4, i)$, 
$\partial d(i, {\bf E}) = 
a(i, 5) 
+ \mu a(i, 1) 
- a(i, 4) 
a(4, 5)$, 

\noindent 
$\partial c({\bf F}, i) = 
\mu a(1, i) 
+ a(4, i) 
- a(1, 5) 
a(5, i)$, 
$\partial d(i, {\bf F}) = 
a(i, 1) 
+ \mu a(i, 4) 
- a(i, 5) 
a(5, 1)$, 

\noindent 
where $i \in \{ 1, 2, 3, 4, 5 \}$. 
A direct calculation proves the theorem. 
\end{proof}

Let $T^n(2, 3)$ denote a 2-sphere 
in ${\mathbb{R}}^4$ 
that is obtained from $T(2, 3)$ 
by the twist-spinning construction, 
introduced by Zeeman \cite{zeeman}, 
where $n \in {\mathbb{Z}}$. 
Inoue \cite{inoue} exhibits an explicit construction of 
a diagram $D^n(2, 3)$ of an oriented 
marked graph representing $T^n(2, 3)$. 
The lower diagram in Figure \ref{example} illustrates 
a diagram $D^2(2, 3)$ of 
an oriented marked graph representing 
$T^2(2, 3)$. 
A direct calculation 
shows the following theorem.

\begin{theorem} \label{calculation-2twist-spun-trefoil} 
The 0-dimensional Yoshikawa homology of $T^2(2, 3)$, 
$HY_0(T^2(2, 3))$, is isomorphic to 
${\mathbb{Z}}[\mu, \mu^{-1}] \langle x, y \rangle / 
(\mu (1 + \mu) + x - y x, 
(1 + \mu) + \mu y - y x, 
\mu (1 + \mu) + \mu^{-1} x - \mu^{-1} x y x, 
- (1 + \mu) + x y - \mu x, 
\mu (1 + \mu) + y - x y, 
y + \mu (1 + \mu) - y (- \mu (1 + \mu) + x y), 
\mu (1 + \mu) + y - \mu^{-1} (- (1 + \mu) + x y) (- \mu (1 + \mu) + x y), 
(1 + \mu) + \mu x - \mu^{-1} (- (1 + \mu) + x y) (- \mu (1 + \mu) + x y))$.
\end{theorem} 

\vskip 6pt

\noindent 
{\it Proof of Theorem \ref{yoshikawa-example}}. 
Theorem \ref{calculation-spun-trefoil} shows that 
the algebra map 
$HY_0(T^0(2, 3)) \otimes {\mathbb{Z}}/3{\mathbb{Z}} 
\to {\mathbb{Z}}/3{\mathbb{Z}}$ 
is realized by three triples 
$(\mu, x, y) = (1, 2, 2), (2, 0, 0), (2, 2, 1)$. 
Theorem \ref{calculation-2twist-spun-trefoil} shows that 
$HY_0(T^2(2, 3)) \otimes {\mathbb{Z}}/3{\mathbb{Z}} 
\to {\mathbb{Z}}/3{\mathbb{Z}}$ 
is realized by two triples 
$(\mu, x, y) = (1, 2, 2), (2, 0, 0)$. 
Since the numbers of algebra maps are distinct, 
$HY_0(T^0(2, 3))$ is not isomorphic to $HY_0(T^2(2, 3))$. 
\hfill $\qed$

\end{document}